\documentclass{amsart}

\usepackage{hyperref}
\usepackage{amsmath}
\usepackage{amsfonts}

\newcommand\nn{\mathbb N}
\newcommand\zz{\mathbb Z}
\newcommand\rr{\mathbb R}
\newcommand\qq{\mathbb Q}
\newcommand\cc{\mathbb C}

\newcommand\oo{\mathcal O}
\def\epsilon{\varepsilon}
\def\phi{\varphi}
\newcommand{\ds}{\mathbb D}
\newcommand{\SH}{\operatorname{SH}}
\newcommand{\PSH}{\operatorname{PSH}}
\newcommand{\inter}{\operatorname{int}}

\theoremstyle{plain}
\newtheorem{theorem}{Theorem}
\newtheorem{proposition}[theorem]{Proposition}
\newtheorem{lemma}[theorem]{Lemma}
\newtheorem{corollary}[theorem]{Corollary}

\theoremstyle{definition}

\theoremstyle{remark}
\newtheorem{remark}[theorem]{Remark}

\numberwithin{equation}{section}
\numberwithin{theorem}{section}

\begin{document}

\title[Dimension of the Bergman space of Hartogs domains]%
{A remark on the dimension of the Bergman space of some Hartogs domains}
\author{Piotr Jucha}
\address{Institute of Mathematics, Jagiellonian University, \L{}ojasiewicza~6,
30--348 Krak\'ow, Poland}
\email{Piotr.Jucha@im.uj.edu.pl}
\thanks{The work is supported by the research grant No.~N~N201~361436 of the Polish Ministry 
of Science and Higher Education.}

\begin{abstract}
Let $D$ be a Hartogs domain of the form
$D=\{(z,w)\in \mathbb C \times \mathbb C^N : \|w\| < e^{-u(z)} \}$ where $u$ is a
subharmonic function on $\mathbb C$. 
We prove that the Bergman space $L^2_h(D)$ of holomorphic and square integrable functions on $D$
is either trivial or infinite dimensional.
\end{abstract}

\date{\today}
\maketitle

\section{introduction}

Let $L^2_h(\Omega)$ denote the Bergman space of a domain $\Omega\subset \cc^N$,
i.e.\ the space of square integrable and holomorphic functions on $\Omega$.

We are interested in the following open question (see e.g. \cite{bib:jar-pfl}, \cite{bib:pfl-zwo}):
\textit{is there a pseudoconvex domain with finite dimensional and nontrivial
Bergman space?}

J.~Wiegerinck (cf.~\cite{bib:wie}) gave examples of Reinhardt domains
in $\cc^2$ such that their Bergman spaces were finite dimensional but nontrivial.
Those domains, however, are not pseudoconvex.
What is more, there exists a simple geometric characterization of pseudoconvex Reinhardt domains:
if a logarithmic image of such a domain contains a real affine line then its Bergman space is $\{0\}$,
otherwise it is infinite dimensional (cf.~\cite{bib:zwo1, bib:zwo2}).

It is known that a Bergman space for any subdomain of $\cc$ is also either
infinite dimensional or trivial (cf.~\cite{bib:skw}, \cite{bib:wie}).

We consider Hartogs domains $D_\phi$ of the form
\[
D_\phi = D_\phi(G) = \{(z,w)\in G\times\cc^N:\, \|w\| < e^{-\phi(z)}\} \subset \cc^M\times\cc^N,
\]
where $G$ is a domain in $\cc^M$, $\phi\in\PSH(G)$ and
$\|\cdot\|$ denotes the maximum norm.
We use the maximum norm for convenience but the same 
results hold for any $\cc$--norm.
(If $\widetilde D_\phi$ is such a domain defined for other $\cc$--norm then
one just needs to take $D_{\phi+c_1} \subset \widetilde D_{\phi} \subset D_{\phi+c_2}$
for suitable constants $c_1,c_2$.)

 We believe that the answer for this question is negative, at least for Hartogs domains. 
 Even though in this paper we are dealing with domains with one dimensional basis, 
 we think that the main idea (cf.~Proposition~\ref{prop:sko_one}) and some techniques of the proofs 
 could also be used at least in some multi-dimensional cases with the help of advanced pluripotential theory.

The main result of the paper is Theorem~\ref{thm:main}, which states that $L_h^2(D_\phi(\cc))$
is either trivial or infinite dimensional.
More precisely, necessary and sufficient conditions on $\phi$ are given for 
$\dim L^2_h(D_\phi(\cc)) = \infty$ and for $\dim L^2_h(D_\phi(\cc)) = 0$.

The method used provided supplementary results:
$\dim L^2_h(D_\phi(G)) = \infty$ if $G \subset \cc^M$ is bounded (Corollary~\ref{prop:bdd})
or $G\subset \cc$ has nonpolar complement (cf.~Corollary~\ref{cor:nonpolar}).

In accordance with a basic result for 
Hartogs domains (Lemma~\ref{lem:hartogs}),
to determine the dimension of the Bergman space
we may consider only the square integrable holomorphic functions on $D_\phi$
of the form $f(z)w^n$ where $f\in\oo(G)$.
Therefore, our strategy in the sequel is to find infinitely many such functions 
(for different $n\in\zz_+^N$)
or, respectively, prove that none of them exists except the zero function.

\begin{lemma}\label{lem:hartogs}
Let $D\subset G\times \cc^{N}$ be a Hartogs domain over $G\subset \cc^M$ with complete 
$N$--circled fibers.
\begin{enumerate}
\item (cf.~\cite{bib:jak-jar})
If $f\in\oo(D)$, then there exist $f_n\in\oo(\cc)$, $n\in\zz_+^N$, such that
\begin{equation}
f(z,w) = \sum_{n\in\zz_+^N} f_n(z)w^n, \qquad (z,w)\in D,
\label{eq:hartogs}
\end{equation}
and the series is locally uniformly convergent.
\item\label{en:hartogs2}
If $f\in L^2_h(D)$, then $f_n(z)w^n \in L^2_h(D)$ for $n\in\zz_+^N$ and
the series {\rm (\ref{eq:hartogs})} is convergent in $L^2_h(D)$.
\end{enumerate}
\end{lemma}

\begin{proof}[Proof of Lemma~\ref{lem:hartogs}~(\ref{en:hartogs2})]
Let $(K_j)_{j\ge 1}$ be a sequence of compacts subsets of $G$ such that it exhausts $G$, 
$K_j \subset \inter K_{j+1}$, and let $L_j:= K_j\times \ds(0,j)^N \cap D$.
Then $\inter L_j$ are Hartogs domains with $N$--circled fibers that exhaust $D$.

The functions $f_n(z)w^n$ are pairwise orthogonal on $\inter L_j$
and the series (\ref{eq:hartogs}) is convergent in $L^2_h(\inter L_j)$.
Therefore, we have for any $j\ge 1$
\[
\int_D |f|^2\, d\lambda^{2(M+N)} \ge \int_{L_j} |f|^2\, d\lambda^{2(M+N)}
= \sum_{n\in\zz_+^N} \int_{L_j} |f_n(z)w^n|^2\, d\lambda^{2(M+N)}(z,w),
\]
where $\lambda^{2(M+N)}$ denotes $2(M+N)$--dimensional Lebesgue measure.
Taking $j \to\infty$, we finish the proof.
\end{proof}

The proofs rely heavily on the properties of subharmonic functions and their singularities,
and we use advanced $L^2$--extension techniques.
We remark, nevertheless, that in some cases (e.g.\ when $G=\cc$ and $\phi$ has logarithmic growth) 
one can explicitly find (infinitely many) functions $f_n(z)w^n$
or, respectively, prove that no such function other than zero exists.
These functions $f_n$ are simply polynomials with zeroes and degrees determined by
Corollary~\ref{cor:sh} (cf.~Proposition~\ref{prop:fin}).

\section{Singularities of subharmonic functions}

For a function $u$ subharmonic in a neighborhood of $a\in\cc$ put
\begin{align*}
m(u,a,r) &:= \frac 1{2\pi} \int_0^{2\pi} u(a+re^{it})\, dt,\\
M(u,a,r) &:= \max_{|z|=r} u(z).
\end{align*}

We define the Lelong number of $u$ at $a$ as
\[
\nu(u,a) := \lim_{r\to 0}\frac {M(u,a,r)}{\log r}.
\]
It is well known (see e.g.\ \cite{bib:ran}) that
\[
\lim_{r\to 0}\frac {M(u,a,r)}{\log r} = \lim_{r\to 0}\frac {m(u,a,r)}{\log r} = 
\frac 1{2\pi} \Delta u(\{a\}) \in [0,+\infty),
\]
where $\frac 1{2\pi}\Delta u$ denotes the Riesz measure of the function $u$.

\begin{proposition}\label{thm:siu}
Let $u$ be a subharmonic function on a domain $D\subset \cc$.
Then for every $\delta>0$ the set
\[
\{z\in D:\, \nu(u,z) \ge \delta\}
\]
is finite.
\end{proposition}

\begin{proof}
The Riesz measure $\frac 1 {2\pi}\Delta u$ is finite on every compact set
$K\subset D$. Therefore, a set $K\cap \{z\in D:\, \nu(u,z) \ge \delta\}$
must be finite.
\end{proof}

Recall the definition of the integrability index of $u$ at $a$
\begin{align*}
\iota(u,a) &:=\inf I_{u,a},\\
\text{where } I_{u,a} &:= \{t>0:\, e^{-\frac {2u}{t}} \text{ is integrable in some neighborhood of } a\}.
\end{align*}
It is clear that $e^{-\frac {2u}{s}}$ is integrable for every $s>\iota(u,a)$, so $I_{u,a}$
is an interval.

The equality of the Lelong number and the integrability index for subharmonic functions 
is a classical result. However,
we did not find any direct reference and 
we give the proof for the sake of completeness.
\begin{proposition}\label{thm:sko}
Let $u$ be a subharmonic function on a neighborhood of $a\in\cc$. Then
\[
\iota(u,a)=\nu(u,a).
\]
Moreover, $e^{-2u}$ is integrable in a neighborhood of $a$ if and only if $\nu(u,a)<1$.
\end{proposition}

\begin{proof}
Let us assume that $a=0$ and $u$ is subharmonic on a neighborhood of the closure
of the disc $\ds(r) = \{|\zeta|<r\}$ with $0<r<\frac 12$.
Then $u$ can be decomposed (cf. Theorem~3.7.9 in \cite{bib:ran}) as
\[
 u(z) = h(z) + \int_{\ds(r)} \log|z-\zeta| \, d\mu(\zeta), \quad z \in \ds(r),
\]
where $h$ is a bounded harmonic function on $\ds$ and $\mu = \frac 1{2\pi} \Delta u|_{\ds(r)}$.
Therefore, there exists a constant $C>0$ such that
\begin{align*}
 u(z)
 &\le C + \int_{\ds(r) \cap \{ |z-\zeta| \le |z|\}} \log|z-\zeta|\, d\mu(\zeta)\\
 &\le C + \mu(\ds(r) \cap \{ |z-\zeta| \le |z|\}) \log|z| \le C + \nu(u,0) \log|z|.
\end{align*}
For this reason, if $e^{-\frac {2u}t}$ is integrable in a neighborhood of $0$ then
$t>\nu(u,0)$, which yields inequality $\nu(u,0) \le \iota(u,0)$.

To prove the other inequality, take numbers $t>\nu(u,0)$ and $r<\frac 12$ so small that
$\mu(\ds(r)) < t$. By the Jensen inequality 
we obtain that
\begin{align*}
 \exp\left(-\frac{2u(z)}t \right) 
 &\le C' \exp\left(-\frac {2 \mu(\ds(r))} {t} \int_{\ds(r)} \log|z-\zeta|\, \frac {d\mu(\zeta)}{\mu(\ds(r))} \right)\\
 &\le C'' \int_{\ds(r)} |z-\zeta|^{-\frac {2 \mu(\ds(r))}{t}}\, d\mu(\zeta),
\end{align*}
for some constants $C',C''>0$. We have then
\[
 \int_{\ds(r)} e^{-\frac{2u}t }\, d\lambda^2
 \le C'' \int_{\ds(r)} \left(  \int_{\ds(r)} |z-\zeta|^{-\frac {2 \mu(\ds(r))}{t}}\,d\lambda^2(z) 
                       \right) \, d\mu(\zeta) < +\infty,
\]
and therefore, $t>\iota (u,0)$.
This finishes the proof.
\end{proof}

Both Proposition~\ref{thm:siu} and Proposition~\ref{thm:sko}
have their highly nontrivial multidimensional  counterparts by, respectively, Y.--T.~Siu and H.~Skoda
(for references and discussion on singularities of plurisubharmonic functions see e.g.~\cite{bib:kis1, bib:kis2}).

\begin{proposition}\label{prop:sh2} {\ }
\begin{enumerate}
\item\label{en:sh2:1}
Let $u\in\SH(\ds_*)$ be such that $A=\liminf_{r\to 0} \frac {M(u,0,r)}{\log r} > -\infty$.
Then
\begin{enumerate}
\item $A<+\infty$ and $u_0(z):= u(z) - A\log |z|$ is subharmonic on $\ds$;
\item $e^{-2u}$ is integrable in a neighborhood of $0$ if and only if $A<1$.
\end{enumerate}
\item\label{en:sh2:2}
Let $u\in\SH(\cc\setminus\ds)$ be such that $\limsup_{|z|\to \infty} \frac {u(z)}{\log|z|} =B$.
\begin{enumerate}
\item If $B<+\infty$, then  $C:=\limsup_{|z|\to \infty} u(z) - B\log |z| < +\infty$.
\item If $B<+\infty$, then
$e^{-2u}$ is integrable in a neighborhood of $\infty$ if and only if $B>1$.
\item  If  $\lim_{|z|\to \infty} \frac {u(z)}{\log|z|} =+\infty$ 
then $e^{-\frac {2u} t}$ is integrable in a neighborhood of $\infty$ for every $t>0$.
\end{enumerate}
\end{enumerate}
\end{proposition}

\begin{proof}(\ref{en:sh2:1}) 
First, observe that if $A>0$, then $u$ extends to a subharmonic function on $\ds$ as 
a nonpositive function in a neighborhood of $0$. Therefore, $A=\nu(u,a) <+\infty$.

For any $\epsilon >0 $ define the function
\[
u_\epsilon(z):= u(z) - A \log|z| + \epsilon \log|z|, \quad z\in\ds_*.
\]
Notice that $u_\epsilon (z) \le \tfrac 12 \epsilon\log|z|$ near $0$,
and hence, it can be extended to a subharmonic function on $\ds$.
The functions $u_\epsilon$ are uniformly bounded from above in a neighborhood of $0$ because
for every $\epsilon>0$ we have 
\[
u_\epsilon (z) < M(u,0,\tfrac 12) +  A\log 2, \quad |z| = \tfrac 12.
\]
Therefore, a function $u_0(z) = \lim_{\epsilon} u_\epsilon(z)$, $0<|z|<1$, is also bounded 
from above near $0$ and is subharmonic on $\ds$.

Statement (ii) follows from Proposition~\ref{thm:sko} if $A\ge 0$.

If $A<0$ then we apply the same argument to the subharmonic function $u(z)-A\log|z|$ 
to obtain that $e^{-2u(z)}|z|^{2A}$, and hence also $e^{-2u}$, is integrable around $0$.

(\ref{en:sh2:2}) 
Notice that $\limsup_{|z|\to \infty} \frac {u(z)}{\log|z|} = 
\limsup_{r\to \infty} \frac {M(u,0,r)}{\log r}$, and therefore,
$\widetilde u(z) := u(\tfrac 1z)$ satisfies (\ref{en:sh2:1})
with $A=-B$ provided that $B<+\infty$. 
Then the function $\widetilde u(z) + B\log|z|$ is subharmonic on $\ds$ and 
$C$ is its value at $0$.

If $\lim_{|z|\to \infty} \frac {u(z)}{\log|z|} =+\infty$ then $e^{-u(z)} < \tfrac 1{|z|^B}$ 
for every $B>0$ and sufficiently large $|z|$.
\end{proof}

\begin{corollary}\label{cor:sh}Let $k\in\zz$.
\begin{enumerate}
\item Let $u$ be a subharmonic function in a neighborhood of a point $a$. Then
$|z-a|^{2k}e^{-2u(z)}$ is integrable in a neighborhood of $a$ if and only if
\[
k> \nu(u,a) -1.
\]
\item Let $u$ be a subharmonic function on $\{|z|>R\}$ for some $R>0$
such that either $\limsup_{|z|\to \infty} \frac {u(z)}{\log|z|} =B<+\infty$ 
or $\lim_{|z|\to \infty} \frac {u(z)}{\log|z|} =B =+\infty$.
Then $|z|^{2k}e^{-2u(z)}$ is integrable in some neighborhood of $\infty$ if and only if
\[
k< B-1.
\]
\end{enumerate}
\end{corollary}
\begin{proof}
Apply Proposition~\ref{prop:sh2} to the function $u-k\log|z|$.
\end{proof}

\section{$L^2$--tools}
\subsection{H\"ormander--Bombieri--Skoda theorem}

The main tool we are going to use in the sequel is Proposition~\ref{prop:sko_one} which follows from 
theorem of Skoda
(a stronger version of Theorem~4.4.4 in \cite{bib:hor}).

\begin{theorem}[\cite{bib:sko2}]\label{thm:sko2}
Let $u$ be a plurisubharmonic function in a pseudoconvex domain $G \subset \cc^M$.
If $e^{-2u}$ is integrable in a neighborhood of a point $z_0\in G$, 
then for any $\epsilon>0$ one can find
an analytic function $f$ in $G$ such that $f(z_0)=1$ and 
\begin{equation}\label{eq:thm:sko2}
\int_G \frac {|f(z)|^2}{(1+|z|^2)^{M+\epsilon}} e^{-2u(z)} \, d\lambda^{2M}(z)< \infty.
\end{equation}
\end{theorem}

An immediate consequence of that theorem, even in its weaker version with the exponent
$3M$ instead of $M+\epsilon$,
is the following.

\begin{corollary}\label{prop:bdd}
Let $G\subset \cc^M$ be a bounded pseudoconvex domain and $\phi\in\PSH(G)$.
Then $\dim L_h^2(D_\phi) = \infty$.
\end{corollary}

\begin{proof}
It is enough to find for every $n\in\zz_+^N$
a function $f_n\in\oo(G)$ not identically equal to $0$ and
such that $f_n(z)w^n \in L_h^2(D_\phi)$. 
Then the sequence $(f_n(z)w^n)_n$ is a set of infinitely many linearly independent elements 
of $L_h^2(D_\phi)$.
Fix $n\in\zz_+^N$ and apply Theorem~\ref{thm:sko2} to the function $u=(N+|n|)\phi$.
Then there is a function $f_n\in\oo(G)$ such that
\begin{multline*}
\int_{D_\phi} |f_n(z)|^2 |w^n|^{2} \, d\lambda^{2(M+N)}(z,w)\\
< C \int_G \frac {|f_n(z)|^2}{(1+|z|^2)^{3M}} e^{-2(N+|n|)\phi(z)}\, d\lambda^{2M}(z) < +\infty,
\end{multline*}
where the constant $C$ depends on $M$, $N$, $n$, and $G$.
\end{proof}

\begin{proposition}\label{prop:sko_one}
Let $v$ be a subharmonic function on a domain $G\subset \cc$.
Suppose that there exists a compact subset $K\subset\subset G$  and a function $u\in\SH(G)$, 
$u\not\equiv -\infty$, such that for some $\epsilon>0$ and $C\in \rr$
\begin{align}
u(z) + (1+\epsilon)\log^+|z| \le C + v(z),\quad &z \in G\setminus K,\label{eq:sko_one:1}\\
\nu(u,z) \ge [\nu(v,z)], \quad &z \in K.\label{eq:sko_one:2}
\end{align}
Then there exists $f\in\oo(G)$, $f\not\equiv 0$, such that
\[
\int_G|f|^2 e^{-2v}\, d\lambda^2 <+\infty.
\]
\end{proposition}

Here, $[x]$ denotes the largest integer not greater than $x\in\rr$.

\begin{proof}
We apply Theorem~\ref{thm:sko2} to the function $u$ and get the function $f\in\oo(G)$ 
not identically equal to $0$, which  satisfies (\ref{eq:thm:sko2}).
Due to condition (\ref{eq:sko_one:1}), we get that
\[
\int_{G\setminus K} |f|^2 e^{-2v} \, d\lambda^2 
  \le \widetilde C \int_{G\setminus K} 
    \frac {|f(z)|^2}{(1+|z|^2)^{1+\epsilon}} e^{-2u(z)} \, d\lambda^2(z) < +\infty,
\]
for some constant $\widetilde C$.

The function $|f|^2e^{-2u}$ is integrable in a neighborhood of any point $z_0\in K$, and
in virtue of (\ref{eq:sko_one:2}) and Corollary~\ref{cor:sh}
the function $|f|^2e^{-2v}$ is integrable in some, possibly smaller, neighborhood of $z_0$.
Using the compactness argument we get the integrability of $|f|^2e^{-2v}$ on $K$,
which finishes the proof.
\end{proof}

\subsection{Ohsawa--Takegoshi extension theorem}

We shall use the Ohsawa--Take\-go\-shi theorem to prove that
the space $L_h^2(D_\phi)$ (for $D_\phi \subset G\times\cc^N$) 
is infinite dimensional provided that
the base domain $G\subset \cc$ has nonpolar complement.
We quote the theorem in simple setting with zero weights but with $D$ possibly unbounded.

\begin{theorem}(cf.~\cite{bib:ohs}, \cite{bib:din})\label{thm:ohs}
Let $(z_0,w_0)\in D\subset G\times \cc^N$ where $G\subset \cc$ is a domain with nonpolar complement.
Then there exists a constant $C>0$ depending only on the domain $G$ such that
for any $f \in L_h^2(D\cap(\{z_0\}\times\cc^N))$ we can find 
$F\in L_h^2(D)$ with $F(z_0,\cdot)=f$ and
\[
\int_D |F|^2\, d\lambda^{2N+2} \le C \int_{D\cap(\{z_0\}\times \cc^N)} |f|^2\, d\lambda^{2N}.
\]
\end{theorem}

\begin{corollary}\label{cor:nonpolar}
If $G \subset \cc$ has nonpolar complement and $\phi\in \SH(G)$, $\phi\not\equiv -\infty$,
then $\dim L_h^2(D_\phi)= \infty$.
\end{corollary}

\begin{proof}
It suffices to take $z_0\in G$ such that $\phi(z_0)>-\infty$.
Then $D_\phi \cap \{z_0\}\times \cc^N$ has infinite dimensional Bergman space for 
it is bounded, and hence, we get the conclusion.
\end{proof}

We remark that there is another consequence of Theorem~\ref{thm:ohs} 
in \cite{bib:din} determining the dimension
of the Bergman space for some pseudoconvex domains in $\cc^2$.
In particular, it applies to a Hartogs domain $D_\phi \subset G \times\cc$,
in the case when either $\cc\setminus G$ is nonpolar or
$\phi$ is bounded from below on $G$.

\section{Hartogs domains with one dimensional bases}

We use the following notation: $[x]$ is the integral part of a number $x\in \rr$
and $\{x\}:= x - [x]$ is the fractional part of $x$.

Let $G\subset \cc$ be a domain and $\phi\in\SH(G)$.
In virtue of Proposition~\ref{thm:siu}, the set $\{a\in G:\, \nu(\phi,a)>0 \}$ is at most
countable.
For that reason we may decompose $\Delta \phi$ as
\begin{equation}\label{eq:decomp}
\Delta \phi = \sum_{j\ge 1} \alpha_j \delta_{a_j} + \mu,
\end{equation}
where $\delta_{a_j}$ are the Dirac measures at some (distinct) points $a_j\in G$,
$\alpha_j = \nu(\phi,a_j)>0$, and $\mu$ is a nonnegative measure equal zero 
on countable sets.

For such a decomposition consider the following condition on weights $\alpha_j$:
\begin{equation}\label{eq:cond}
\begin{split}
&\exists\ j_1\neq j_2:\
\{\alpha_{j_1}\}, \{\alpha_{j_2}\}, \{\alpha_{j_1}+ \alpha_{j_2}\} >0\\
\text{or }
&\exists\ j_1 < j_2 <j_3:\  
  \{\alpha_{j_1}\}, \{\alpha_{j_2}\}, \{\alpha_{j_3}\}, \{\alpha_{j_1} + \alpha_{j_2} + \alpha_{j_3}\} >0.
\end{split}
\end{equation}
In other words,
condition~(\ref{eq:cond}) says that there exist at least two non-integral 
weights such that their sum is not an integer.
In fact, the only case when we need more than two weights in the statement of this condition 
is $\alpha_{j_1} = \alpha_{j_2} = \alpha_{j_3} = \frac 12$.

\begin{theorem}\label{thm:main}
Let $\phi\in\SH(\cc)$ and suppose that we have decomposition (\ref{eq:decomp})
for $\phi$.
If $\mu\not\equiv 0$ or condition~(\ref{eq:cond}) is satisfied, then
$\dim L^2_h(D_\phi) = \infty$.
Otherwise, $L^2_h(D_\phi) = \{0\}$.
\end{theorem}

\begin{proof}
To prove the first part of the theorem it suffices to collect the results from
Proposition~\ref{prop:num} and
Proposition~\ref{prop:notharmonic} below.

Suppose that $\mu\equiv 0$ and condition~(\ref{eq:cond}) does not hold,
i.e.
$\Delta\phi = \sum_{j\ge 1} \alpha_j\delta_{a_j}$ and at most two
of numbers $\alpha_j$, say $\alpha_1$, $\alpha_2$, are not natural.
Take any $F\in L^2_h(D_\phi)$. Without loss of generality we may assume that
$F(z,w) = f(z)w^n$ for some $n\in\zz_+^N$.
Thus,
\begin{equation}\label{eq:int}
\int_\cc |f|^2e^{-2\psi}\, d\lambda^2 <+\infty,
\end{equation}
where $\psi=(N+|n|)\phi$. 
There exists $g\in\oo(\cc)$ such that
\[
\psi(z) = \log|g(z)| + \alpha_1^\prime \log|z-a_1| + \alpha_2^\prime \log|z-a_2|, \quad z \in G,
\]
with $\alpha_j^\prime = \{(N+|n|)\alpha_j\}$, $j=1,2$,
and $g$ having zeroes at $a_j$ of order $[(N+|n|)\alpha_j]$ for $j\ge 1$.

Therefore, $f$ must have zeroes at these points of at least the same order 
(cf.~Corollary~\ref{cor:sh}) and $h:=\frac fg$ extends to an entire holomorphic function.
It follows from (\ref{eq:int}) that
\[
\int_{|z|>R} |h(z)|^2|z|^{-2(\alpha_1^\prime+\alpha_2^\prime)}\, d\lambda^2(z) < +\infty.
\]
Expanding $h$ in a Taylor series and using Corollary~\ref{cor:sh} 
with $u=(\alpha_1^\prime+\alpha_2^\prime)\log|z|$
(recall that $\alpha_1^\prime+\alpha_2^\prime \le 1$) we obtain
that $\lim_{|z|\to\infty}|h(z)|=0$.
Thus, both $h$ and $f$ must be identically equal to zero, and in consequence, $L_h^2(D_\phi)=\{0\}$.
\end{proof}

\begin{proposition}\label{prop:num}
Let a domain $G\subset \cc$ and a function $\phi\in\SH(G)$ be such that
there exist $z_j \in G$ with $\alpha_j:=\nu(\phi,z_j)$, $j=1,2,3$
satisfying
\[
\{\alpha_1+\alpha_2+\alpha_3\},\{\alpha_1\},\{\alpha_2\}>0, \alpha_3\ge 0.
\]
Then $\dim L^2_h(D_\phi)=\infty$.
\end{proposition}

\begin{proof}
By Lemma~\ref{lem:number} there exist infinitely many $k\in \nn$ such 
that $\sum_{j=1}^3 \{k\alpha_j\} > 1$.
Fix a multi-index $n\in\zz$ such that $N+|n|=k$ for some $k\in\nn$ and $\epsilon>0$ so small that
\[
\sum_{j=1}^3 \{ (N+|n|)\alpha_j \} > 1 + \epsilon.
\]
Define
\[
u(z):= (N+|n|)\phi(z) - \sum_{j=1}^3 \{ (N+|n|)\alpha_j \} \log|z-z_j|,
\]
which is a subharmonic function on $G$. Moreover, there exist $\delta>0$ and $C>0$ such that
\[
u(z) + (1+\epsilon)\log^+|z| \le (N+|n|)\phi(z) + C, 
  \quad z \in G\setminus\bigcup_{j=1}^3\ds(z_j,\delta),
\]
and $\nu(u,z)\ge [\nu((N+|n|)\phi,z)]$ for $|z-z_j|<\delta$.
Therefore, in view of Proposition~\ref{prop:sko_one} there exists $f_n\in\oo(G)$ not identically
equal to zero such that
\[
\int_{D_\phi} |f_n(z)w^n|^2\, d\lambda^{2(1+N)}(z,w) = C_1 \int_G |f_n|^2e^{-2(N+|n|)\phi} \, d\lambda^2<+\infty.
\]
A sequence $f_n(z)w^n$, for suitable multi-indices $n$, is an infinite set of linearly independent
elements of $L^2_h(D_\phi)$.
\end{proof}

\begin{proposition}\label{prop:notharmonic}
Let a domain $G\subset \cc$ and a function $\phi\in\SH(G)$ be such that
there exists a disc $\ds(z_0,\delta) \subset\subset G$ with the following property
\begin{align*}
\Delta \phi &\not\equiv 0 \text{ on } \ds(z_0,\delta),\\
\nu(\phi,z)&=0, \ z \in \ds(z_0,\delta).
\end{align*}
Then $\dim L^2_h(D_\phi)=\infty$.
\end{proposition}

\begin{proof}
We may assume that $\phi(z_0)>-\infty$ and $\phi$ is bounded on $\partial\ds(z_0,\widetilde\delta)$ 
for some $0<\widetilde\delta \le \delta$. Indeed,  since
the set $\{\phi < \phi(z_0) -1 \}$ is thin at $z_0$, it cannot intersect all the circles
centered at $z_0$
(see the proof of Thm.~3.8.3 in \cite{bib:ran}).
To simplify notation we also assume that $\widetilde\delta=1$ and $z_0=0$.

There exist (cf.~Thm.~4.5.4, \cite{bib:ran}) a harmonic function $h$
and a subharmonic function $p$ on $\ds$ such that
$\phi=h+p$ on $\ds$, $h=\phi$ a.e.\ on $\partial \ds$ and $\limsup_{z\to\partial\ds} p(z) = 0$.
The assumption of the proposition provides that $p(0)<0$.
Thus, there exist numbers $\delta^\prime \in(0,1)$ and $A>0$ such that $p\le -3A$
on $\ds(0,\delta^\prime)$.
Define
\[
\widetilde p(z):=
\begin{cases}
\frac {A}{-\log\delta^\prime} \log|z| - A, &\text{if } |z|\le \delta^\prime,\\
\max\big( \frac {A}{-\log\delta^\prime} \log|z| - A, \widehat p(z)\big), &\text{if }  \delta^\prime<|z|<1,
\end{cases}
\]
where $\widehat p(z):= M(p,0,z)$. Certainly, $\widetilde p$ is subharmonic on $\ds$
and equals $0$ on $\partial\ds$.
Let
\[
\widetilde \phi:=
\begin{cases}
\phi &\text{on } G\setminus\ds,\\
\widetilde p + h &\text{on } \ds.
\end{cases}
\]

We claim that $\widetilde \phi$ is subharmonic on $G$. In fact, it suffices
to show that $\widetilde \phi$ is upper semicontinuous on $\partial \ds$.
Note that the following harmonic modification of $\phi$:
\[
\psi=
\begin{cases}
\phi &\text{on } G\setminus\ds,\\
h &\text{on } \ds
\end{cases}
\]
is subharmonic on $G$ (if 
$(\phi_j)_{j\ge 1}$ is a decreasing sequence of continuous subharmonic functions 
tending to $\phi$ on $G$ then $\psi$ is the limit of their harmonic modification).
The inequality $\widetilde \phi \le \psi$ yields the upper semicontinuity of $\widetilde \phi$.

One can see that 
$\nu(\widetilde\phi,z)=0$ for $z\in\ds\setminus\{0\}$
and $\nu(\widetilde\phi,0)>0$.

Let $n\in\zz_+^N$ be such that $(N+|n|)\nu(\widetilde\phi,0)>1$.
Then we can apply Proposition~\ref{prop:sko_one} to $v=(N+|n|)\phi$ and 
$u(z)=(N+|n|)\widetilde\phi(z) - (N+|n|)\log|z|$. 
(Observe that $u\in\SH(\cc)$ because of Proposition~\ref{prop:sh2}.)
Therefore, similarly as in Proposition~\ref{prop:num} we obtain an infinite sequence of 
linearly independent members of $L^2_h(D_\phi)$.
\end{proof}

\begin{lemma}{\label{lem:number}}
Let $\alpha_1, \alpha_2, \alpha_3$ be real numbers such that $0\le \alpha_3 \le \alpha_2 \le \alpha_1 <1$.
If any of the following conditions is satisfied
\begin{enumerate}
\item\label{en:num1} $\alpha_2 = \alpha_3 =0$,
\item\label{en:num2} $\alpha_1 + \alpha_2 = 1$, $\alpha_3 =0$,
\end{enumerate}
then 
\[
\sum_{j=1}^3 \{ k\alpha_j \} \le 1 \text{ for all } k\in\nn.
\]
Otherwise
\begin{equation}
\sum_{j=1}^3 \{k\alpha_j\} >1 \text{ for infinitely many } k\in\nn. 
\end{equation}
\end{lemma}

\begin{proof}
Suppose that condition (\ref{en:num2}) holds (case (\ref{en:num1}) is trivial).
Then for any $k\in \nn$ we have
\[
\{k\alpha_1\} + \{k\alpha_2\} = \{k\alpha_1\} + \{k -k\alpha_1 \} =
\{k\alpha_1\} + \{ -k\alpha_1 \} \le 1.
\]

To prove the other part of the statement we examine several cases.
Note that if all $\alpha_j$'s are positive and not all of them equal $\frac 12$, 
then we may assume (in Cases~2--5) that 
\[
\alpha_1, \alpha_2, \{\alpha_1 + \alpha_2\} >0. 
\]
Moreover, we do not require $\alpha_1 \ge \alpha_2$.

\emph{Case 1.} $\alpha_1 = \alpha_2 = \alpha_3 = \frac 12$. For any odd $k$ we have
\[
 \{k\alpha_1\} + \{ k\alpha_2\} + \{k\alpha_3\} = \frac 32.
\]

\emph{Case 2.} $\alpha_1, \alpha_2 \in\qq$.

There are $p_1,p_2,q \in \nn$ such that $\alpha_j=\tfrac{p_j}{q}$, $j=1,2$.
If $p_1 + p_2  <q$, take $k:= lq -1$ for any $l\in\nn$ and notice that
\[
\{k\alpha_1\} + \{k\alpha_2\} = \left\{- \frac {p_1}q \right\} +  \left\{- \frac {p_2}q \right\}
= 1 - \frac {p_1}q + 1 - \frac {p_2}q >1.
\]
If $p_1 + p_2 > q$, the inequality holds for $k:=lq +1$, $l \in \nn$.

\emph{Case 3.} $\alpha_1\in \rr\setminus \qq$, $\alpha_2\in\qq$.

Let $p,q\in\qq$ be such that $\alpha_2 = \frac pq$. By the Kronecker theorem
(see e.g.~\cite{bib:apo}), the set
$\big\{\{(lq+1)\alpha_1\}:\, l\in\zz\big\}$ is dense in $(0,1)$, and therefore,
we can find infinitely many $k\in \nn$ of the form $k=lq+1$ 
such that $\{k\alpha_1\} >1 - \frac {p}{q}$.
Thus, for such numbers $k$ we have 
\[
\{k\alpha_1\} + \{k\alpha_2\} > 1 - \frac pq  + \left\{ (lq+1) \frac pq \right\} 
=1.
\]

\emph{Case 4.} $\alpha_1,\alpha_2 \in \rr\setminus \qq$ 
and $1,\alpha_1,\alpha_2$ are linearly independent over $\zz$ 
(i.e.\ if $s_1\alpha_1 + s_2\alpha_2 \in\zz$ for some $s_1,s_2\in\zz$, then $s_1=s_2=0$).

By the two-dimensional Kronecker theorem (cf.~\cite{bib:apo}) the set 
$\big\{(\{k\alpha_1\},\{k\alpha_2\}):\, k\in\zz \big\}$ is dense in $(0,1)^2$.
Therefore, there are infinitely many $k\in\nn$ satisfying $\{k\alpha_j\}> \frac 12$, $j=1,2$.
For such $k$ we have
$\{k\alpha_1\}+ \{k\alpha_2\} > 1$.

\emph{Case 5.} $\alpha_1, \alpha_2 \in\rr\setminus \qq$ and $1,\alpha_1,\alpha_2$ 
are linearly dependent over $\zz$.

Let $p,q, r\in\zz$ be such that $p\alpha_1+ q\alpha_2 = r$.

If $pq>0$, $p\neq q$,  we may assume that $0<p<q$ and 
take any $l\in \nn$ such that $\{l\alpha_1\} < \frac 1{pq}$
(there is infinitely many such $l\in\nn$ by the Kronecker theorem). Then
we have $\{ p l\alpha_1\} < \{q l\alpha_1\}$, and consequently, for $k=ql$,
\[
\{k \alpha_1\} + \{k \alpha_2\} = \{q l \alpha_1\} + \{ - p l\alpha_1\}
= \{q l \alpha_1\} + 1 - \{p l \alpha_1\} > 1.
\]

Suppose that $p=q>0$. 
We may find, again by the Kronecker theorem, infinitely many $l\in\nn$ such that
$\left\{ \frac rp \right\} < \{(pl+1)\alpha_1\} < \left\{ \frac rp \right\} + \frac 1{p}$.
Therefore, if $k = pl+1$,
\[
\{k\alpha_1\} + \{k\alpha_2\}  
= \{(pl+1)\alpha_1\} + \left\{ \frac rp - (pl+1)\alpha_1\right\}
> \left\{ \frac rp \right\} + \frac {p-1}{p} > 1.
\]

Finally, if
$pq< 0$, assume that $0<-q<p$. There are infinitely many $l\in \nn$ such that 
$\{l\alpha_2\} \in (\frac {1}{p+1}, \frac 1p)$.
For $k = pl$ we get
\[
 \{ k \alpha_1\} + \{k \alpha_2 \} = \{ -q l \alpha_2\} + \{pl \alpha_2 \}
 > \frac {-q}{p+1} + \frac p {p+1} \ge 1.
\]
This finishes the proof.
\end{proof}

There are some special cases when we can avoid using advanced $L^2$--techniques
to prove Theorem~\ref{thm:main}
and, moreover, describe the space $L^2_h(D_\phi)$.
This is, for instance, when $\phi$ is a subharmonic function of logarithmic growth on $\cc$.

\begin{proposition}\label{prop:fin}
Let $\phi\in\SH(\cc)$ be such that
$\limsup_{|z|\to\infty} \frac {\phi(z)}{\log |z|} = \gamma$
for some $\gamma\ge 0$ and suppose that $\Delta \phi$ has decomposition (\ref{eq:decomp}).
Let $n\in\zz_+^N$. Then the following conditions are equivalent:
\begin{enumerate}
\item\label{en:fin:1}
there exists $f_n\in\oo(\cc)$, $f_n\not\equiv 0$, such that $f_n(z)w^n \in L_h^2(D_\phi)$;
\item\label{en:fin:2}
$\sum_{j\ge 1} [ (N+|n|)\alpha_j] < (N+|n|)\gamma-1$.
\end{enumerate}
Moreover, if such a function $f_n$ exists, it is a polynomial of degree 
smaller than $(N+|n|)\gamma -1$.
\end{proposition}

\begin{proof}
(\ref{en:fin:1})$\implies$(\ref{en:fin:2})

Due to Proposition~\ref{prop:sh2}, we have 
$\phi(z) \le \gamma \log^+|z| + C$ for $z\in\cc$ and some constant $C>0$.
Therefore, it is clear that
\[
\widetilde D := 
  \left\{(z,w)\in \cc^{1+N}:\, \|w\| < e^{-C}\min(1, |z|^{-\gamma}) \right\}
\subset D_\phi.
\]
The domain $\widetilde D$ is complete Reinhardt,
hence we have the expansion
$f_n(z)w^n=\sum_{k=0}^\infty a_kz^kw^n$.
What is more, the monomials $a_kz^kw^n$ are pairwise orthogonal members of $L_h^2(\widetilde D)$.
It follows  from Corollary~\ref{cor:sh} applied to $u= (N+|n|)\gamma\log|z|$ that  $a_k = 0$ if
$k\ge (N+|n|)\gamma -1$.

On the other hand, the finiteness of the integrals
\[
\int_{D_\phi} |f_n(z)|^2|w|^{2n} \, d\lambda^{2(1+N)}(z,w) 
= C_n \int_{\cc} |f_n(z)|^2 e^{-2(N+|n|)\phi(z)} d\lambda^2(z)
\]
implies, again by Corollary~\ref{cor:sh}, that $f_n$ must have zeroes at points $a_j$ 
of multiplicity at least $[(N+|n|)\alpha_j]$.
The sum of these multiplicities does not exceed the degree of the polynomial $f_n$.

(\ref{en:fin:2})$\implies$(\ref{en:fin:1})

Denote $k_j:=[(N+|n|)\alpha_j]$ and let $l\in\nn$ be such that
\[
\sum_{j=1}^\infty [(N+|n|)\alpha_j] = \sum_{j=1}^l k_j < (N+|n|)\gamma -1.
\]
Put
\[
f_n(z):= \prod_{j=1}^l (z-a_j)^{k_j}, \quad z\in\cc,
\]
and use Corollary~\ref{cor:sh} to verify that $f_n(z)w^n\in L_h^2(D_\phi)$.
\end{proof}

\begin{remark}
If $\phi$ is as in Proposition~\ref{prop:fin} then we have the well known fact:
$\sum_{j\ge 1} \alpha_j \le \Delta \phi(\cc) \le \gamma$.
Moreover, if $\sum_{j\ge 1} \alpha_j = \gamma$ then inequality (\ref{en:fin:2})
is equivalent to $\sum_{j\ge 1} \{\alpha_j\}>1$.
Thus, Proposition~\ref{prop:fin} and Lemma~\ref{lem:number} give another proof 
of Theorem~\ref{thm:main} in this special case.
\end{remark}

\begin{remark}
 The results of this paper
 (Theorem~\ref{thm:main}, Proposition~\ref{prop:num}, Proposition~\ref{prop:notharmonic}, 
 and Corollary~\ref{cor:nonpolar}) do not yet give full answer to the problem
 of the dimension $L_h^2(D_\phi)$ of a Hartogs domain $D_\phi$ with one dimensional basis.
 The case where the basis $G$ has polar complement
and the function $\phi\in\SH(G)$ is harmonic remains unsolved.

The starting point of the proof in this setting may be the following. Suppose that
$L_h^2(D_\phi)$ is not trivial, then 
there exists $f\in\oo(G)$ not identically equal to $0$ and such that 
$\int_G |f|^2 e^{-2\phi} \, d\lambda^2 <+\infty$.
Define a function $h:= \phi - \log|z|$, which is harmonic on $G\setminus f^{-1}(0)$, and moreover,
$e^{-2h}$ is integrable.

If $\cc \setminus G$ is finite or if $h$ can be decomposed as a difference of two functions
subharmonic on $\cc$, the proof follows the same line as in this paper. However, in general
setting better understanding of singularities of harmonic functions is needed.
\end{remark}

\textbf{Acknowledgements.}
The author would like to thank to his professors, M.~Jarnicki and W.~Zwonek, for 
helpful discussions and remarks.


\begin{thebibliography}{Pfl--Zwo}

\bibitem[Apo]{bib:apo}
T.~M.~Apostol, \emph{Modular functions and Dirichlet Series in Number Theory},
Springer-Verlag, 1976.





\bibitem[Din]{bib:din}
\.Z.~Dinew, \emph{The Ohsawa--Takegoshi extension theorem on some unbounded sets},
  Nagoya Math.~J. 188 (2007), 19--30.

\bibitem[H\"or]{bib:hor}
L.~H\"ormander, \emph{An Introduction to Complex Analysis in Several Variables},
  North Holland, 1990.

\bibitem[Jak--Jar]{bib:jak-jar}
P.~Jak\'obczak, M.\ Jarnicki, \emph{Lectures of holomorphic functions of several complex
variables}, \texttt{http://www.im.uj.edu.pl/MarekJarnicki/}.

\bibitem[Jar--Pfl]{bib:jar-pfl}
M.~Jarnicki, P.~Pflug, \emph{Invariant Distances and Metrics in Complex Analysis},
  de Gruyter Exp.\ Math.\ 9, de Gruyter, Berlin, 1993.

\bibitem[Kis~1]{bib:kis1}
C.~Kiselman, \emph{Densit\'e des fonctions plurisousharmoniques},
  Bull.\ Soc.\ Math.\ France 107 (1979), 295--304.

\bibitem[Kis~2]{bib:kis2}
C.~Kiselman, \emph{Attenuating the singularities of plurisubharmonic functions},
  Ann.\ Pol.\ Math.\ 60 no. 2 (1994), 173--197.



\bibitem[Ohs]{bib:ohs}
T.~Ohsawa, \emph{On the extension of $L^2$ holomorphic functions III: negligible weights},
  Math.~Z., 219(1995), 145--148.

\bibitem[Pfl-Zwo]{bib:pfl-zwo}
P.~Pflug, W.~Zwonek, \emph{$L_h^2$-domains of holomorphy in the class of unbounded Hartogs domains},
 Illinois J.\ Math.\ 51 (2007), no.~2, 617--624.

\bibitem[Ran]{bib:ran}
T.~Ransford, \emph{Potential Theory in the Complex Plane},
  Cambridge University Press, 1995.


\bibitem[Sko]{bib:sko2}
H.~Skoda, \emph{Estimations $L^2$ pour l'op\'erateur $\overline\partial$ et applications arithm\'etique},
  Lect.\ Notes in Math.\ 578 (1977), 314--323.

\bibitem[Skw]{bib:skw}
M.~Skwarczy\'nski, \emph{Evaluation functionals on spaces of square integrable holomorphic functions},
  Prace Matematyczno--Fizyczne, M.~Skwarczy\'nski, W.~Wasilewski editors, 
  Wy\.zsza Szko\l{}a In\.zynierska w Radomiu, Radom, 1982.


\bibitem[Wie]{bib:wie}
J.~Wiegerinck, \emph{Domains with finite dimensional Bergman space},
  Math.\ Zeitschrift 187 (1984), 559--562.

\bibitem[Zwo~1]{bib:zwo1}
W.~Zwonek, \emph{On Bergman completeness of pseudoconvex Reinhardt domains},
  Ann.\ Fac.\ Sci.\ Toulouse, VI. S\'er.\ Math.\ 8 (1999), 537--552 (no.\ 3).


\bibitem[Zwo~2]{bib:zwo2}
W.~Zwonek, \emph{Completeness, Reinhardt domains and the method of complex 
  geodesics in the theory of invariant functions},
  Dissertationes Math.\ 388 (2000).

\end{thebibliography}
\end{document}